\documentclass[12pt]{amsart} 
\usepackage{amscd}
\usepackage{amssymb}
\usepackage{a4wide}
\usepackage{amstext}
\usepackage{amsthm}
\usepackage{xcolor}
\numberwithin{equation}{section}

\newcommand{\CC}{\mathbb {C}}

\newcommand{\RR}{\mathbb{R}}

 \DeclareMathOperator{\dist}{dist}

\renewcommand{\phi}{\varphi}
\newcommand{\rea}{{\rm Re}\,}

\newcommand{\he}{\mathcal{H}(T,A, \mu)}
\newcommand{\ho}{\mathcal{H}}

\newtheorem{Thm}{Theorem}[section]
\newtheorem{theorem}[Thm]{Theorem}
\newtheorem{example}{Example}
\newtheorem{lemma}[Thm]{Lemma}
\newtheorem{proposition}[Thm]{Proposition}
\newtheorem{corollary}[Thm]{Corollary}
\newtheorem{remark}[Thm]{Remark}
\newtheorem{definition}{Definition}

\begin{document}
\sloppy
\title[Localization of zeros in Cauchy--de Branges spaces]{Localization of zeros 
in Cauchy--de Branges spaces}
\author{Evgeny Abakumov, Anton Baranov, Yurii Belov}
\address{Evgeny Abakumov,
\newline  University Paris-Est,  LAMA (UMR 8050), UPEM, UPEC, CNRS, F-77454, Marne-la-Vall\'ee, France
\newline {\tt evgueni.abakoumov@u-pem.fr}
\smallskip
\newline \phantom{x}\,\, Anton Baranov,
\newline Department of Mathematics and Mechanics, 
St.~Petersburg State University, St.~Petersburg, Russia,
\newline {\tt anton.d.baranov@gmail.com}
\smallskip
\newline \phantom{x}\,\, Yurii Belov,
\newline   Department of Mathematics and Computer Science, St.~Petersburg State University, St.~Petersburg, Russia,
\newline {\tt j\_b\_juri\_belov@mail.ru}
}

\begin{abstract} We study the class of discrete measures in the complex plain
with the following property: up to a finite number, 
all zeros of any Cauchy transform of the measure (with $\ell^2$-data) 
are localized near the support of the measure. 
We find several equivalent forms of this property and prove that 
the parts of the support attracting zeros of Cauchy transforms are ordered
by inclusion modulo finite sets. 
\end{abstract}


\thanks{The work was supported by the joint grant of Russian Foundation for Basic Research 
(project 17-51-150005-NCNI-a) and CNRS, France 
(project PRC CNRS/RFBR 2017-2019 ``Noyaux reproduisants dans
des espaces de Hilbert de fonctions analytiques''). }

\maketitle

\hfill {\it \small Dedicated to the memory of Serguei Shimorin,}  

\hfill {\it \small a brilliant mathematician and a wonderful person}

\section{Introduction and main results}

Representations of analytic functions via Cauchy transforms of 
planar measures is a classical theme in function  theory. 
Of special interest are expansions of meromorphic functions
as Cauchy transforms of discrete (atomic) measures. A substantial number of papers 
deals with distribution of zeros of such Cauchy transforms. 
Note that zeros of Cauchy transforms are equilibrium points 
of logarithmic potentials for the corresponding discrete measures.   

Let $T=\{t_n\}_{n\in\mathbb{N}}$ 
be a sequence of distinct complex numbers with $|t_n| \to\infty$, $n\to\infty$. 
Then, for any sequence $a=\{a_n\}$ such that $\sum_n |t_n|^{-1} |a_n| <\infty$ 
one can consider the Cauchy transform 
$$
\mathcal{C}_{a}(z) = \sum_{n=1}^\infty \frac{a_n}{z-t_n}.
$$ 
J.~Clunie, A.~Eremenko and J.~Rossi \cite{cer} conjectured that 
the Cauchy transform $\mathcal{C}_{a}$ has 
infinitely many zeros if all $a_n$ are positive. 
In the general case this conjecture remains open,  
for related results see \cite{cer, el, lr}. 
Clearly, if the coefficients $a_n$ are not positive, the corresponding 
sum can be the inverse to an entire function and, thus, can have no zeros. 

In \cite{abb} we studied some phenomena connected with the following 
heuristic principle: {\it if the coefficients $a_n$ are {\it extremely small}, then
all \textup(except a finite number of\textup)  
zeros of $\mathcal{C}_{a}$ are located near the set $T$ or near its part}. 
We called this {\it localization property}. In \cite{abb} only the case 
$T\subset \mathbb{R}$ was considered. This enabled us to relate the problem 
with the theory of de Branges spaces and structure of Hamiltonians for canonical systems.
In the present paper we consider the case of general (complex) $t_n$ 
and extend many of the results from \cite{abb} to this setting. One of the 
motivations for this study is the role of discrete Cauchy transforms 
in the functional model for rank one perturbations of compact normal operators
(see \cite{by3, b2018}).


\subsection{The spaces of Cauchy transforms and related spaces of entire functions} 
Let $\mu:=\sum_n\mu_n\delta_{t_n}$ be a positive measure on
$\mathbb{C}$ such that
$\sum_n \frac{\mu_n}{|t_n|^2 +1} < \infty$. With
any such $\mu$ we associate the Hilbert space $\mathcal{H}(T,\mu)$ of the
Cauchy transforms
$$
\mathcal{H}(T,\mu):=\biggl{\{}f:f(z)=\sum_n\frac{a_n\mu^{1/2}_n}{z-t_n},
\quad a = \{a_n\}\in\ell^2\biggr{\}}
$$
equipped with the norm
$\|f\|_{\mathcal{H}(T,\mu)}:=\|a\|_{\ell^2}$.
Note that the series in the definition 
of $\mathcal{H}(T,\mu)$ converges absolutely and uniformly on compact sets 
which do not  intersect $T$.

The spaces $\mathcal{H}(T,\mu)$ consist of meromorphic functions
which are analytic in $\mathbb{C}\setminus T$. To get rid of 
the poles, we will usually consider isometrically isomorphic
Hilbert spaces of entire functions.
Let $A$ be an entire function which has only simple zeros and whose zero set 
$\mathcal{Z}_A$ coincides with $T$. With any $T$, $A$ and $\mu$ as above
we associate the space $\mathcal{H}(T,A,\mu)$ of entire functions,
$$
\mathcal{H}(T,A,\mu):=\biggl{\{}F:F(z)= A(z)\sum_n\frac{a_n\mu^{1/2}_n}{z-t_n},
\quad a = \{a_n\}\in\ell^2\biggr{\}},
$$
where again the norm is given by 
$\|F\|_{\mathcal{H}(T,A,\mu)}:=\|a\|_{\ell^2}$.
Clearly, the mapping $f\mapsto Af$ is a unitary operator from  
$\mathcal{H}(T,\mu)$ to $\mathcal{H}(T,A,\mu)$.
We will use the term {\it Cauchy--de Branges spaces} for the 
spaces $\mathcal{H}(T,A,\mu)$. This is related to the fact that
the class of the spaces $\mathcal{H}(T,A,\mu)$ with $T\subset\RR$ 
coincides with the class of all de Branges spaces (see \cite{br}).

The spaces $\mathcal{H}(T,A,\mu)$ were introduced in full generality 
by Yu. Belov, T. Mengestie,
and K. Seip \cite{BMS}. They can also be described axiomatically. 
It is clear that the reproducing kernels of $\he$ at the points $t_n$
(which are of the form $\overline{A'(t_n)} \mu_n \cdot \frac{A(z)}{z-t_n}$)
form an orthogonal basis in $\he$. Conversely, if 
$\mathcal{H}$ is a reproducing kernel Hilbert space
of entire functions such that
\smallskip
\begin{enumerate} 
\item [(i)]
$\mathcal{H}$ has the {\it division property}, that is, $\frac{f(z)}{z-w} 
\in \mathcal{H}$ whenever $f\in\mathcal{H}$ and $f(w) = 0$,
\smallskip
\item [(ii)]
there exists an orthogonal basis of reproducing kernels in $\mathcal{H}$,
\end{enumerate}
\smallskip
then $\mathcal{H} = \mathcal{H}(T,A,\mu)$ for some choice of the parameters $T$, $A$ and $\mu$.
One can replace existence of an orthogonal basis by existence of a Riesz basis 
of normalized reproducing kernels. In this case $\mathcal{H}$ coincides with
some space $\mathcal{H}(T,A,\mu)$ as sets with equivalence of norms.


\subsection{Localization and strong localization} 
To simplify certain formulas, we will always assume in what follows that 
$$
|t_n| \ge 2, \qquad t_n\in T.
$$
Also we will always assume that $T$ is 
{\it a power separated sequence}: there exist numbers $C>0$ and
$N>0$ such that, for any $n$,
\begin{equation}
\label{powsep}
{\rm dist}\,(t_n, \{t_m\}_{m\ne n}) \geq C|t_n|^{-N}.
\end{equation}
Note that condition \eqref{powsep} implies that for some $c,\rho>0$
and for sufficiently large $n$ we have $|t_n| \ge cn^\rho$.
We always will choose $A$ to be an entire function of finite order 
with zeros at $T$.
Without loss of generality we may always fix $N$ in \eqref{powsep}
so large that $\sum_n |t_n|^{-N} < \infty$.

For an entire function $f$ we
denote by $\mathcal{Z}_f$ the set of all zeros of $f$.  
Let $D(z,r)$ stand for the open disc centered at $z$ of radius $r$.

Now we introduce the notion of {\it zeros localization}.

\begin{definition} 
We say that the space $\mathcal{H}(T,A,\mu)$
with a power separated sequence $T$ has the localization property if 
there exists a sequence of disjoint disks $\{D(t_n,r_n)\}$ 
with $r_n\to 0$ such
that for any nonzero $f\in \mathcal{H}(T,A,\mu)$ the set
$\mathcal{Z}_f\setminus\cup_n D(t_n,r_n)$ is finite and each disk
$D(t_n,r_n)$ contains at most one point of $\mathcal{Z}_f$ for any
$n$ except, possibly, a finite number.
\end{definition}

Since the space $\mathcal{H}(T,A, \mu)$ has the division property, 
one can construct a function from $\mathcal{H}(T,A,\mu)$ 
with zeros at any given finite set. Therefore the notion of localization
of the zeros near $T$ makes sense only up to finite-dimensional sets.

Our first result shows that the localization property in $\mathcal{H}(T,A,\mu)$
can be expressed in several natural ways. For a set $E$, we denote by $\# E$ 
the number of elements in $E$.

\begin{theorem}
\label{local}
Let $\mathcal{H}(T,A,\mu)$ be a Cauchy--de Branges space with a 
power separated $T$. The following statements are equivalent:
\begin{enumerate}
\begin{item}
$\mathcal{H}(T,A,\mu)$ has the localization property\textup; 
\end{item}
\begin{item}
There exists an unbounded set $S\subset \mathbb C$ such that the set
$\mathcal{Z}_f\cap S$ is finite for any nonzero
$f\in\mathcal{H}(T,A,\mu)$\textup;
\end{item}
\begin{item}
For any $f\in \mathcal{H}(T,A,\mu) \setminus\{0\}$ and $M>0$ we have
$\#\big(\mathcal{Z}_f \setminus \cup_n D(t_n, |t_n|^{-M})\big) <\infty$\textup;
\end{item}
\begin{item}
There is no nonzero $f\in\mathcal{H}(T,A,\mu)$ with infinite number
of multiple zeros.
\end{item}
\end{enumerate}
\end{theorem}

Similarly to \cite{abb} one can introduce the notion of {\it strong localization}
where the zeros are localized only near the whole set $T$.
We say that the space $\mathcal{H}(T,A,\mu)$ with a power separated
sequence $T$ has the {\it strong localization property} if there exists
a sequence of disjoint disks $\{D(t_n,r_n)\}_{t_n\in T}$ 
with $r_n\to 0$ such that for any
nonzero $f\in \mathcal{H}(T,A,\mu)$ 
the set $\mathcal{Z}_f\setminus\cup_n D(t_n,r_n)$ is finite
and each disk $D(t_n,r_n)$ contains
exactly one point of $\mathcal{Z}_f$ for any $n$ except, possibly,
a finite number.

As in \cite{abb}, one can show that the strong
localization property is equivalent to the approximation by
polynomials.

\begin{theorem}
\label{strlocal}
The space $\mathcal{H}(T,A,\mu)$ has the strong localization
property if and only if the polynomials belong to
$L^2(\mu)$ and are dense there.
\end{theorem}

Note that polynomials belong to $L^2(\mu)$ whenever 
$\mathcal{H}(T,A,\mu)$  has the localization property 
(see Proposition \ref{mu_n}). Density of polynomials in weighted $L^p$ spaces 
is a classical problem in analysis studied by M.~Riesz, S.~Bernstein, N.~Akhiezer, 
S.~Mergelyan, L.~de~Branges, and many
others (see, e.g., \cite{Akh, Merg, Koo1, BS0, BS1}). 
All these works treat the case when the measure in question
is supported by the real line. For measures in $\mathbb{C}$ 
the problem seems to be largely open. 
\medskip


\subsection{Attraction sets}
\label{AS}
Let $\mathcal{H}(T,A,\mu)$ have the localization property. By the
property (iii) from Theorem \ref{local}, with any nonzero
function $f\in\mathcal{H}(T,A,\mu)$ we may associate a set
$T_f\subset T$ such that for some disjoint disks $D(t_n,r_n)$
all zeros of $f$ except, may be, a finite number
are contained in $\cup_{t_n\in T} D(t_n,r_n)$ and 
there exists exactly one point of $\mathcal{Z}_f$ in each disk
$D(t_n,r_n)$, $t_n\in T_f$, except, may be, a finite number of indices $n$.
Thus, the set $T_f$ is uniquely defined by $f$ up to finite sets.
Let us also note that we can always take $r_n=|t_n|^{-M}$ for
any $M>0$. 

\begin{definition}
Let $\mathcal{H}(T,A,\mu)$ have  the localization property. We will
say that $S\subset T$ is an attraction set if there exists
$f\in\mathcal{H}(T,A,\mu)$ such that $T_f=S$ up to a finite set.
\end{definition}

Note that $f(z) = \frac{A(z)}{z-t_0} \in \mathcal{H}(T,A,\mu)$ 
for any $t_0\in T$, and so $T$ is always an attraction set.

It turns out that the localization property implies the following
ordering theorem for the attraction sets of $\mathcal{H}(T,A,\mu)$.

\begin{theorem}
Let $\mathcal{H}(T,A,\mu)$ be a Cauchy--de Branges space with
the localization property. Then for any two attraction sets $S_1$,
$S_2$ either $S_1\subset S_2$ or $S_2\subset S_1$ up to finite
sets. 
\label{ordering}
\end{theorem}

This ordering rule has some analogy with the de Branges Ordering
Theorem for the chains of de Branges subspaces.  For the case of de Branges spaces
(i.e., $T\subset\RR$) the ordering structure of attraction sets was proved in 
\cite[Theorem 1.8]{abb}. 
Two different proofs were given: one of them used the de Branges Ordering Theorem,
while the other used only a variant of Phragm\'en--Lindel\"of principle
due to de Branges \cite[Lemma 7]{br}, a deep result which is one of the main
steps for the Ordering Theorem. 

These methods are no longer available in the case of nonreal $t_n$. However,
it turned out that one can give a completely elementary proof 
of Theorem \ref{ordering}, independent of de Branges' Lemma. 
Thus, we can essentially simplify the proof of \cite[Theorem 1.8]{abb}.
\medskip


\subsection{Localization of type $N$}
\label{AST}
We say that the space $\mathcal{H}(T,A,\mu)$ has
{\it the localization property of type $N$} if there exist  $N$
subsets $T_1$, $T_2$,...,$T_N$ of $T$ such that $T_j\subset T_{j+1}$, $1\leq j\leq N-1$,
$\#(T_{j+1}\setminus T_j)=\infty$ and for any nonzero
$f\in\mathcal{H}(T,A,\mu)$ we have $T_f=T_j$ for some $j$, $1\le j\le N$, 
up to  finite sets, moreover, $N$ is the smallest integer with this property.
Clearly, in this case $T_N=T$ up to a finite set. 
The strong localization is the localization of type $1$.

In what follows we say that an entire function $F$ of finite order 
is in the {\it generalized Hamburger--Krein class}
if $F$ has simple zeros $\{z_n\}$, 
\begin{equation}
\label{krein0}
\lim_{n\to\infty} |F'(z_n)|^{-1}|z_n|^M = 0 \qquad \text{for any}\qquad M>0,
\end{equation}
and
\begin{equation}
\label{krein}
\frac{1}{F(z)} = \sum_n \frac{1}{F'(z_n)(z - z_n)}.  
\end{equation}
Note that when $\{z_n\}$ is power separated and \eqref{krein0} is satisfied,
one can replace \eqref{krein} by the condition that, for any $K>0$, 
$|F(z)|\gtrsim 1$ when $z\notin \cup_n D(z_n, (|z_n|+1)^{-K})$.

Now we state the description of spaces with localization property of type $2$;
localization of type $N$ can be described similarly (see \cite[Theorem 6.1]{abb}).

\begin{theorem}
 \label{type2}
The space $\mathcal{H}(T,A,\mu)$ has the localization property of
type $2$ if and only if there exists a partition $T=T_1\cup T_2$,
$T_1\cap T_2=\emptyset$, such that the following three conditions
hold:
\begin{enumerate}
\begin{item}
There exists an entire function $A_2$ in the generalized 
Hamburger--Krein class such that $\mathcal{Z}_{A_2}=T_2$;
\end{item}
\begin{item}
The polynomials belong to the space $L^2(T_2,{\mu }|_{T_2})$, 
are not dense there, but their 
closure is of finite codimension in $L^2(T_2,{\mu }|_{T_2})$.
\end{item}
\begin{item}
The polynomials belong to the space $L^2(T_1,\tilde{\mu})$ and are dense there,
where $\tilde{\mu}=\sum_{t_n\in T_1}\mu_n|A_2(t_n)|^2\delta_{t_n}$.
\end{item}
\end{enumerate}
{Moreover, $T_1$ and $T$ are the attraction sets for $\mathcal{H}(T,A,\mu)$.}
\end{theorem}

In Section \ref{exx} we will give a number of examples 
of Cauchy--de Branges spaces having localization property of type $2$. 
\bigskip
 

\section{Equivalent forms of zeros localization \label{locsec}}
In this section we will prove Theorem \ref{local}.

A sequence $\{z_k\}\subset \CC$ will be said to be {\it lacunary} 
if $\inf_k |z_{k+1}|/|z_k| >1$. A zero genus canonical product over a lacunary 
sequence will be said to be a {\it lacunary canonical product}.

For $\Omega\subset \CC$, we define its {\it upper area density} by 
$$
D^+(\Omega) = \limsup_{R\to\infty} \frac{m_2(\Omega \cap D(0, R))}{\pi R^2},
$$
where $m_2$ denotes the area Lebesgue measure in $\CC$. If $D^+(\Omega) = 0$ we say that
$\Omega$ is a set of zero area density. 
We say that a set $E\subset\RR$ has zero linear density if 
$|E \cap (0, R)| = o(R)$, $R\to\infty$, where 
$|e|$ denotes one-dimensional Lebesgue measure of $e$.

The following result (\cite[Theorem 2.6]{abb1}) will play an important role in what follows. 
We will often need to verify that a certain entire function belongs 
to the space $\mathcal{H}(T,A,\mu)$.  
In the de Branges space setting a much stronger statement is given 
in \cite[Theorem 26]{br}. 

\begin{theorem}
\label{inc}
Let $\mathcal{H}(T,A,\mu)$ be a Cauchy--de Branges space
and let $A$ be of finite order. Then an entire function $f$ 
is in $\mathcal{H}(T,A,\mu)$ if and only if the following three conditions hold:
\begin{enumerate} 
\item [(i)] 
$\sum_n\dfrac{|f(t_n)|^2}{|A'(t_n)|^2 \mu_n} <\infty$\textup;      
\smallskip
\item [(ii)]
there exist a set $E\subset (0,\infty)$ 
of zero linear density and $N>0$ such that $|f(z)| \le |z|^N |A(z)|$,
$|z| \notin E$\textup;
\smallskip
\item [(iii)] there exists a set $\Omega$ of positive upper area density such that 
$|f(z)| = o(|A(z)|)$, $|z|\to \infty$, $z\in \Omega$.
\end{enumerate}
\end{theorem}

Condition (iii) can be replaced by a stronger conditions that 
$|f(z)| = o(|A(z)|)$ as $|z|\to \infty$ outside a set of zero area density.
It should be mentioned that (iii) is a consequence
of the following standard fact about planar Cauchy transforms
(see, e.g., \cite[Proof of Lemma 4.3]{bbb-fock}). If $\nu$ is a finite complex Borel measure
in $\CC$, then, for any $\varepsilon>0$, 
there exists a set $\Omega$ of zero area density such that
$$
\bigg|\int_\mathbb{C}\frac{d\nu(\xi)}{z-\xi} - \frac{\nu(\mathbb{C})}{z} \bigg|
< \frac{\varepsilon}{|z|}, \qquad  z\in\CC\setminus\Omega.
$$
We will frequently use the following corollary of this fact:
if $\nu$ is orthogonal to all polynomials (meaning that $\int |\xi|^k d|\nu|(\xi)<\infty$ 
and $\int \xi^k d\nu(\xi) =0$, $k\in \mathbb{Z}_+$), then for any $M>0$ we have
\begin{equation}
\label{or}
\bigg|\int_\mathbb{C}\frac{d\nu(\xi)}{z-\xi}\bigg| = o(|z|^{-M})
\end{equation}
as $|z|\to \infty$ outside some set of zero area density. 
\medskip

\begin{proof}[Proof of Theorem \ref{local}]
It is obvious that the localization property implies each of the conditions
(ii) and (iv). We will show that (ii)$\Longrightarrow$(iii),
(iii)$\Longrightarrow$(iv), and
(iii) \& (iv) \bigskip         
$\Longrightarrow$(i).

{\bf (ii) $\Longrightarrow$ (iii)}.
Assume that (iii) is not true. Then
for some $M>0$ there exists a nonzero function $F \in \mathcal{H}$
for which there exists an infinite number of zeros
$z\in \mathcal{Z}_F$ with $\dist(z, T)\ge |z|^{-M}$. 

Let $S$ be an unbounded set which satisfies (ii).
Then we can choose two sequences $s_k\in S$
and $z_k \in \mathcal{Z}_F$ such that $2|z_k| \le |s_k| \le |z_{k+1}|/2$
and $\dist(z_k, T)\ge  |z_k|^{-M}$. Now put
$$
H(z) = F(z) \prod_k \frac{1-z/s_k}{1-z/z_k}.
$$
A simple estimate of the lacunary infinite products implies that $|H(z)| 
\lesssim |z|^{M+1}|F(z)|$ for $|z|\ge 1$
and $\dist(z, \{z_k\})\ge |z_k|^{-M}/2$, in particular, for $z\in T$.
Now dividing $H$ by some polynomial $P$ of degree  $M+1$
with $\mathcal{Z}_F \subset \mathcal{Z}_F \setminus\{z_k\}$,
we conclude by Theorem \ref{inc}  that $\tilde H = H/P$ is in $\mathcal{H}$.
This contradicts (ii) since $\mathcal{Z}_{\tilde H} \cap S$
is  an infinite set.   
\bigskip

{\bf (iii) $\Longrightarrow $(iv)} 
Assume that (iv) is not true. Then there exist a nonzero 
function $F\in\he$ and a sequence $\{z_n\}$ of its multiple zeros. 
Without loss of generality we may assume that the sequence $z_k$ is lacunary. 
By (iii) there exists a sequence $n_k$ such that $|z_k - t_{n_k}| =o(|t_{n_k}|^{-K})$ 
for any $K>0$ as $k\to \infty$. Now put 
$$
\tilde F(z) = F(z) \prod_k \frac{(z-t_{n_k} - |t_{n_k}|^{-M})(z-t_{n_k})}{(z-z_k)^2}
$$
for some sufficiently large $M>N$, where $N$ is the constant in \eqref{powsep}.
It is easy to see that $|\tilde F(z)|\lesssim |F(z)|$ when $z\notin \cup_n D(t_n,
C|t_n|^{-N}/2)$ and so $\tilde F$ is in $\he$ by Theorem \ref{inc}.
\bigskip

{\bf (iii) \& (iv) $\Longrightarrow$ (i)}. 
Let $F$ be a nonzero function in $\he$. By (iii), all zeros of $F$, 
except a finite number, are localized in the disks $D(t_n, |t_n|^{-M})$ for any fixed
$M$. Assume that an infinite subsequence of disks $D(t_{n_k}, |t_{n_k}|^{-M})$ 
(where $t_{n_k}$ is a lacunary sequence) contains two zeros $z_k$, $\tilde z_k$ of $F$.
Then the function
$$
\tilde F(z) = F(z) \prod_k\frac{(z-t_{n_k})^2}{(z-z_k)(z-\tilde z_k)}
$$
is in $\he$ by Theorem \ref{inc}, a contradiction with (iv). 
\end{proof}
\bigskip


\section{Localization and polynomial density}

This section is devoted to the proof of Theorem \ref{strlocal}.
In Subsection \ref{pd1} we show that 
the polynomial density implies the strong localization property. 
In Subsection \ref{pd2} we will prove the converse statement.

First of all we prove that  the localization property implies that $\mu_n$ decrease superpolynomially.

\begin{proposition} 
\label{mu_n}
Let $\mathcal{H}(T,A,\mu)$ have the localization property. Then for any $M>0$
we have $\mu_n\lesssim |t_n|^{-M}$.
\end{proposition}

\begin{proof}
Assume the converse. Then there exists $M>0$ and an 
infinite subsequence $\{n_k\}$ such that $\mu_{n_k}\geq |t_{n_k}|^{-M}$.
Without loss of generality we can assume that $\{t_{n_k}\}$ is lacunary. 
Let $U$ be the lacunary product with zeros $t_{n_{10k}}$. Put  
$$
f(z) = A(z)U^3(z)\prod_{k}\biggl{(}1-\frac{z}{t_{n_k}}\biggr{)}^{-1}.
$$
Then, by simple estimates of lacunary canonical products, we have 
$|f(t_{n_k})| \lesssim |t_{n_k}|^{-K} |A'(t_{n_k})|$  for any fixed $K>0$, 
whence $f$ satisfies condition (i) of Theorem \ref{inc}. 
Since $\frac{A(z)}{z-t_{n_1}} \in \he$ and 
$|U(z)|^3 \lesssim |z|^{-K} \prod_{k} |1-z/t_{n_k}|$
for any $K>0$ and $z\notin \cup_k D(t_{n_k}, 1)$, we conclude that $f$ satisfies 
conditions (ii) and (iii) of Theorem \ref{inc} and so $f\in \he$. 
This contradicts the property (iv) from Theorem \ref{local}.
\end{proof}

\subsection{Polynomial Density $\Longrightarrow$ Strong Localization Property\label{pd1}}
Let $f\in\mathcal{H}(T,A,\mu)\setminus\{0\}$. 
If the polynomials are dense in $L^2(\mu)$, then it is not difficult to show
that for any $M>0$ there exist $L>0$ and $R>0$ such that 
\begin{equation}
\inf\{|z|^L|f(z)|: \dist(z,T)\geq|z|^{-M}, |z|>R\}>0.
\label{LReq}
\end{equation} 
A simple proof of this fact is given in detail in 
\cite[Section 3.1]{abb} and we omit it. 

In particular, it follows from \eqref{LReq} that for any $M>0$ all zeros 
of $f\in\mathcal{H}(T,A,\mu)\setminus\{0\}$ except, 
may be, a finite number, are in $\cup_n D(t_n,|t_n|^{-M})$. 
Therefore by Theorem \ref{local}, the space $\mathcal{H}(T,A,\mu)$ 
has the localization property, and so any disc $D(t_n,|t_n|^{-M})$ 
except a finite number contains at most one zero of $f$.

Now we show that the disk $D(t_k,|t_k|^{-M})$ contains 
exactly one point of $\mathcal{Z}_f$ if $|k|$ is sufficiently large. 
Let 
$$
f(z) = \sum_n \frac{d_n\mu_n^{1/2}}{z-t_n}, \qquad
g(z) = \sum_{n\neq k}\frac{d_n\mu_n^{1/2}}{z-t_n}.
$$
Recall that $|f(z)|\geq c|z|^{-L}$ for $|z-t_k|=|t_k|^{-M}$ and sufficiently large 
$k$, where $L$ is the number from \eqref{LReq}.
Since $\mu_k=o(|t_k|^{-\tilde{L}})$, $k\rightarrow\infty$, 
for any $\tilde{L}>0$, we conclude that $|f(z)-g(z)|<c|z|^{-L}/2$
for $|z-t_k|=|t_k|^{-M}$, $k\geq k_0$.

Put $F=Af$, $G=Ag$. Then $F$, $G$ are entire and 
$|F-G|<|G|$ on $|z-t_k|=|t_k|^{-M}$, $k\geq k_0$. 
By the Rouch\'{e} theorem, $F$ and $G$ have the same 
number of zeros in $D(t_k,|t_k|^{-M})$, $k\geq k_0$. 
Since $G(t_k)=0$, we conclude that $F=Af$ has a zero in 
$D(t_k,|t_k|^{-M})$, $|k|\geq k_0$. The strong localization property is proved.

\subsection{Strong Localization $\Longrightarrow$ Polynomial
Density.\label{pd2}} This implication is almost trivial. Let $\{u_n\} \in
\ell^2$ be a nonzero sequence such that $\sum_{n}  u_n t_n^k
\mu_n^{1/2} = 0$ for any $k\in \mathbb{N}_0$. Consider the
function
$$
F(z) = A(z) \sum_n \frac{u_n \mu_n^{1/2} }{z-t_n}.
$$
Then $F$ belongs to the Cauchy--de Branges space $\mathcal{H}(T,A,\mu)$ 
and since all the moments of $u_n$ are zero, it is
easy to see that for any $K>0$, 
$|F(z)/A(z)|=  o(|z|^{-K})$ as $|z|\to\infty$ and 
$z\notin \cup_n D(t_n, C|t_n|^{-N}/2)$, where $C,N$ are parameters from \eqref{powsep}.
On the other hand, since we have the strong localization property,
for any $M>0$ all but a finite number of zeros of $f$ lie in
$\cup_n D(t_n, r_n)$, where $r_n = |t_n|^{-M}$ and
$\#\bigl{(} \mathcal{Z}_f\cap D(t_n,r_n)\bigr{)} \leq 1$ for all indices $n$ except,
possibly, a finite number. 

Let $T_1$ be the set of those $t_n$ for which the corresponding 
disk $D(t_n,r_n)$ contains exactly one zero of $F$ (denoted by $z_n$
with the same index $n$)
and let $A=A_1A_2$ be the corresponding factorization of $A$,
where $A_2$ is a polynomial with finite zero set $T\setminus T_1$. 
Put
$$
F_1(z) = A_1(z)\prod_{t_n \in T_1} \frac{z-z_n}{z- t_n}.
$$
We can choose $M$ to be so large that the above product converges, and, moreover, 
$|F_1(z)|\asymp |A_1(z)|$
when $\dist(z, T_1) \ge C|z|^{-N}/2$. Then we can write $F = F_1F_2$, and it is easy to 
see that in this case $F_2$ is at most a polynomial. 
Thus, for some $L>0$, we have $|F(z)|/|A(z)| \gtrsim |z|^{-L}$, 
as $|z|\to\infty$ and $z\notin \cup_n D(t_n, C|t_n|^{-N}/2)$, a contradiction.
\qed
\bigskip


\section{Ordering theorem for the zeros of Cauchy transforms}

First we show that in the proof of ordering
for attraction sets one can consider only functions with zeros
in $T$.

\begin{lemma}
\label{TfLemma}
Let $f\in\he$, $f\neq0$, and let $T_f$ be defined as in Subsection
\ref{AS}. Then there exists a function $A_f\in\he$ which vanishes
exactly  on $T_f$ up to a finite set. 
\end{lemma}

\begin{proof}
Let $z_n$ be a zero of $f$ closest to the point $t_n\in T_f$.
Since $T_f$ is defined up to finite sets, we may 
assume without loss of generality  that this is a one-to-one
correspondence between $\mathcal{Z}_f$ and $T_f$. Put
$$
A_f(z)=f(z)\prod_{t_n\in T_f}\frac{z-t_n}{z-z_n}.
$$

Since we have $|z_n-t_n|\leq|t_n|^{-M}$ with $M$ much larger than
$N$ from the power separation condition \eqref{powsep}, 
it is easy to see that $|A_f(z)|\asymp|f(z)|$, ${\rm dist}\,(z, T) \ge C|t_n|^{-N}/2$, 
and
$|A_f(t_n)|\asymp|f(t_n)|$, $t_n\in T\setminus T_f$. 
Hence, $A_f\in\he$ by Theorem \ref{inc}. 
\end{proof}                                                      

\begin{corollary}
\label{maxx}
Let $\he$ have the localization property and assume that 
the zeros of a function $f\in\he$ are localized near the whole set $T$ 
up to a finite set. Then for any $K>0$ there exist $c, M>0$ such that
for the discs $D_k = D(t_k, |t_k|^{-K})$, $t_k\in T$, we have
\begin{equation}
\label{yui0}
|f(z)|\ge c|z|^{-M} |A(z)|, \qquad z\notin \cup_k D_k.
\end{equation}  
\end{corollary}

\begin{proof}
Let $A_f$ be a function constructed from $f$ as in Lemma \ref{TfLemma}.
Since the zero set $T_f$ of $A_f$ differs from $T$ by a finite set and all functions in $\he$
are of finite order, we can write
$A_f = A P Q^{-1} e^{R}$, where $P, Q, R$ are some polynomials. 
Let us show that $R$ is a constant. Indeed, since $A/(z-t_n) \in \he$ 
for any $t_n\in T$, we have 
$$
A_f - \frac{A}{z-t_n} = A \bigg(\frac{P}{Q}\, e^{R} - \frac{1}{z-t_n} \bigg) \in \he. 
$$
If $R\ne const$, then the function in brackets will have infinitely many zeros,
a contradiction to localization.

As mentioned above, 
$|A_f(z)|\asymp|f(z)|$, ${\rm dist}\,(z, T) \ge C|t_n|^{-N}/2$, 
where $C, N$  are constants from power separation condition \eqref{powsep}.
This implies \eqref{yui0}.
\end{proof}                                                      
\medskip

Now we pass to the proof of Theorem \ref{ordering}.
By Lemma \ref{TfLemma} we may assume, in what follows, that
$f=A_1$, $g=\tilde{A_1}$, where $\mathcal{Z}_{A_1},
\mathcal{Z}_{\tilde{A_1}}\subset T$. Thus, we may write
$A=A_1A_2=\tilde{A_1}\tilde{A_2}$ for some entire functions $A_2$
and $\tilde{A_2}$.

Let $A_1=BA_0$, $\tilde{A_1}=\tilde{B}A_0$, where $B$ and
$\tilde{B}$ have no common zeros. To prove Theorem \ref{ordering},
we need to show that either $B$ or $\tilde{B}$ has finite number
of zeros.

Note that $A_1 - \alpha \tilde{A_1}$ is in $\mathcal{H}(T, A, \mu)$
for any $\alpha\in \mathbb{C}$. 
Therefore, the zeros of $B - \alpha \tilde B$ are localized near $T$.
As we will see, this is a very strong restriction which cannot hold unless
one of the functions $B$ or $\tilde{B}$ has finite number of zeros.



\subsection{Key proposition} The following proposition is the crucial step 
of the argument. In \cite{abb} a similar statement was proved using a deep result
of de Branges \cite[Lemma 7]{br}; it was valid even without localization 
assumption. This argument is no longer applicable in non-de Brangean case 
when $t_n$ are nonreal. However, taking into account the localization property,
one can give an elementary proof in the general case.                            

\begin{proposition}
\label{ordlemma}
If the functions $B$ and $\tilde B$ defined above have infinitely many zeros, then 
there exists $M>0$ such that at least one of the following two statements holds\textup: 

\begin{enumerate}
\begin{item}
there exists a subsequence $\{t_{n_k}\}\subset \mathcal{Z}_B$ such
that $|B'(t_{n_k})|\leq 8|t_{n_k}|^M|\tilde{B}(t_{n_k})|$\textup;
\end{item}

\begin{item}
there exists a subsequence
$\{t_{n_k}\}\subset\mathcal{Z}_{\tilde{B}}$ such that
$|\tilde{B}'(t_{n_k})|\leq 8|t_{n_k}|^M|B(t_{n_k})|$.
\end{item}
\end{enumerate}
\end{proposition}  

\begin{proof} 
We will often use the following obvious observation: 
if $f$ is a function of finite order, $|z_0|>2$ and 
${\rm \dist}\,(z_0, \mathcal{Z}_f) \ge c |z_0|^{-K}$, then there exists $L>0$ (depending on $K$
and $c$) such that
\begin{equation}
\label{obv}
|f(z_0)|/2 \le |f(z)| \le 2|f(z_0)|, \qquad z\in D(z_0, |z_0|^{-L}).
\end{equation}
This statement follows by standard estimates of canonical products.
\medskip
\\
{\bf Step 1.} 
Assume that neither of the conclusions of the proposition holds. 
Consider the case where $|B'(t_n)| > 8|t_n|^M| \tilde{B}(t_n)|$
for any $M$ and all $n$ except a finite number. It follows from \eqref{obv}
(applied to $B(z)/(z-t_n)$) that there exists $L$ such that $|B(z)| \ge |t_n|^{-L} |B'(t_n)|/2$  
and also $|\tilde B(z)|\le 2|\tilde B(t_n)|$  
for $z \in C_n = \{|z-t_n| = |t_n|^{-L}\}$. Now if $M>L$, we have
$2|\tilde B(z)|<|B(z)|$. By the Rouch\'e theorem, we conclude that 
for any $\alpha$ with $1\le |\alpha|\le 2$, the function 
$B-\alpha\tilde B$ has exactly one zero in the disc $D(t_n, |t_n|^{-L})$, $t_n \in 
\mathcal{Z}_{B}$.  Similarly, 
$B-\alpha\tilde B$ has exactly one zero in the disc $D(t_n, |t_n|^{-L})$, 
$t_n \in  \mathcal{Z}_{\tilde{B}}$.
We conclude that 

{\it for any sufficiently large $L>0$ and any $\alpha$
with $1\le |\alpha|\le 2$,
the function  $B - \alpha\tilde B$ has exactly one zero in each 
disc $D(t_n, |t_n|^{-L})$, $t_n \in \mathcal{Z}_{B} \cup \mathcal{Z}_{\tilde{B}}$,
except   a finite number. }
\medskip
\\
{\bf Step 2.} 
Next we prove the following: {\it there exists an infinite set
$\mathcal{A} \subset \{1\le |z| \le 2\}$ 
such that for any $\alpha \in \mathcal{A}$ and any $L>0$ 
the function $B - \alpha \tilde B$ has at most finite  number of zeros
outside the union of the discs $D_n = D(t_n, |t_n|^{-L})$, 
$t_n \in \mathcal{Z}_{B} \cup \mathcal{Z}_{\tilde{B}}$. }

In view of localization we know that all zeros of $B-\alpha\tilde B$ 
are localized near $T$. Thus, we only need to show that for many values of $\alpha$ the function 
$B-\alpha \tilde B$ has no zeros in a neighborhood of $t_n \in
T\setminus (\mathcal{Z}_{B} \cup \mathcal{Z}_{\tilde{B}})$.
Put
$$
w_n = \frac{B(t_n)}{\tilde B(t_n)}, \qquad
t_n \in T\setminus (\mathcal{Z}_{B} \cup \mathcal{Z}_{\tilde{B}}). 
$$
Since $B, \tilde B$ are entire functions of finite order 
whose zeros are power separated from 
$T\setminus (\mathcal{Z}_{B} \cup \mathcal{Z}_{\tilde{B}})$, 
it follows that 
for any $M>0$ there exists sufficiently large $L>0$ such that 
for $z\in D_n = D(t_n, |t_n|^{-L})$ with 
$t_n \in T\setminus (\mathcal{Z}_{B} \cup \mathcal{Z}_{\tilde{B}})$,
\begin{equation}
\label{cov}
\bigg| \frac{B(z)}{\tilde B(z)}  - w_n\bigg| < |t_n|^{-M}, \qquad   z \in D_n. 
\end{equation}
Obviously, the discs 
$D(w_n, |t_n|^{-M})$, $t_n \in T\setminus (\mathcal{Z}_{B} \cup \mathcal{Z}_{\tilde{B}})$
do not cover the annulus $\{1\le|z|\le 2\}$ if $M$ is sufficiently large. 
Therefore, \eqref{cov} implies that we have a continuum of $\alpha$ 
with $1\le|\alpha|\le 2$ such that $\alpha \ne B(z)/\tilde B(z)$ 
for $z\in \cup_{t_n \in T\setminus (\mathcal{Z}_{B} \cup \mathcal{Z}_{\tilde{B}})} D_n$,
Thus, for such $\alpha$, all zeros of $B-\alpha\tilde B$
up to a finite number belong to $\cup_{t_n \in \mathcal{Z}_{B} \cup \mathcal{Z}_{\tilde{B}}} D_n$
as required.
\medskip
\\
{\bf Step 3.} By Steps 1 and 2, if $\alpha \in \mathcal{A}$, then
all zeros of the function $B-\alpha \tilde B$ are located near 
$\mathcal{Z}_{B} \cup \mathcal{Z}_{\tilde{B}}$. Then we can write
$$
B-\alpha \tilde B = B\tilde B R_\alpha e^{Q_\alpha} \Pi_\alpha, \qquad \Pi_\alpha(z) = 
\prod_{t_n \in \mathcal{Z}_{B} \cup \mathcal{Z}_{\tilde{B}}} \frac{z-s_n}{z-t_n},
$$
where $s_n \in D_n$ are the zeros of $B-\alpha \tilde B$, 
$R_\alpha$ is some rational function and $Q_\alpha$ 
is a polynomial.    
\medskip
\\
{\bf Step 4.} Assume that there exist $\alpha\ne \beta$ such that $Q_\alpha = Q_\beta = Q$. 
Then for $B_1 = e^{Q}B$ and $\tilde B_1 = e^{Q}\tilde B$ we have
$$
B_1-\alpha \tilde B_1 = B_1 \tilde B_1 R_\alpha  \Pi_\alpha, \qquad 
B_1-\beta \tilde B_1 = B_1 \tilde B_1 R_\beta  \Pi_\beta.
$$
It follows that 
$$
\beta-\alpha = B_1 (R_\alpha  \Pi_\alpha - R_\beta  \Pi_\beta), 
\qquad 
\alpha^{-1}-\beta^{-1} = \tilde B_1 (\alpha^{-1} R_\alpha  \Pi_\alpha - \beta^{-1} 
R_\beta  \Pi_\beta).
$$

Since for any fixed $\alpha$ the zeros $s_n$ of $B-\alpha \tilde B$ satisfy
$|s_n-t_n| <|t_n|^{-K}$ for any $K>0$, it is easy to see that $\Pi_\alpha$
admits the expansion 
$$
\Pi_\alpha = 1 +\sum_{k=1}^K  \frac{c_k}{z^k} + O\bigg(\frac{1}{z^{K+1}}\bigg) 
$$
as $|z|\to\infty$, $z\notin \cup_{t_n \in \mathcal{Z}_{B} \cup \mathcal{Z}_{\tilde{B}}}
D_n$ for any fixed $L$. 
Therefore, the function $R_\alpha  \Pi_\alpha - R_\beta  \Pi_\beta$
is either equivalent to $cz^{-K}$ for some $K\in \mathbb{Z}$ or decays faster than any power
when $|z|\to\infty$, $z\notin \cup_{t_n \in \mathcal{Z}_{B} \cup \mathcal{Z}_{\tilde{B}}}
D_n$.

Assume that $B_1$ is not a polynomial. Then $|B_1|$ tends to infinity faster than any power
along some sequence of points outside 
$\cup_{t_n \in \mathcal{Z}_{B} \cup \mathcal{Z}_{\tilde{B}}}
D_n$. In view of the form of $\Pi_\alpha$ 
and $\Pi_\beta$ this implies that for any $K>0$ we have
$$
|R_\alpha  \Pi_\alpha - R_\beta  \Pi_\beta| = o(|z|^{-K}),
\qquad
|z|\to\infty, \ \ z\notin \cup_{t_n \in \mathcal{Z}_{B} \cup \mathcal{Z}_{\tilde{B}}}
D_n.
$$
Similarly, if $\tilde B_1$ is not a polynomial, 
then for any $K>0$,
$$
|\alpha^{-1} R_\alpha  \Pi_\alpha - \beta^{-1} R_\beta  \Pi_\beta| = o(|z|^{-K}),
\qquad |z|\to\infty, \ \ z\notin \cup_{t_n \in \mathcal{Z}_{B} \cup \mathcal{Z}_{\tilde{B}}}
D_n.
$$
Since $\alpha\ne \beta$, it follows that $R_\alpha  \Pi_\alpha $ decays faster 
than any power, a contradiction. Thus, either $B_1$ or $\tilde B_1$ is a polynomial.  
\medskip
\\
{\bf Step 5.} It remains to consider the case when $Q_\alpha \ne Q_\beta$ for any $\alpha, 
\beta\in \mathcal{A}$, $\alpha\ne \beta$. Without loss of generality we may assume 
that there exist $\alpha_0, \alpha_1, \alpha_2, \alpha_3 \in\mathcal{A}$ 
with $Q_{\alpha_j}$ of the same 
degree $m$ such that the coefficients $c_{\alpha_j}$ at $z^m$ are different. 
Dividing by $e^{Q_{\alpha_0}}$ we obtain new functions $B_1$ and $\tilde B_1$ 
satisfying 
$$
\alpha_0 - \alpha_j = B_1 (e^{\tilde Q_j} R_{\alpha_j}\Pi_{\alpha_j} -
R_{\alpha_0}\Pi_{\alpha_0}), \qquad j=1,2,3,
$$
where $R_\alpha$, $\Pi_\alpha$ are defined as above and $\tilde Q_j = 
Q_{\alpha_j} - Q_{\alpha_0}$. Thus, 
$$
|B_1(z)| \asymp |R_{\alpha_j}(z)|^{-1} e^{-{\rm Re}\, \tilde Q_j(z)}, \qquad j=1,2,3, 
$$
as $|z|\to \infty$ along each ray $\{z=re^{i\theta}\}$ 
on which $\lim_{r\to\infty} \rea \tilde Q_j(re^{i\theta}) =\infty$ and 
outside the set $\cup_{t_n \in \mathcal{Z}_{B} \cup \mathcal{Z}_{\tilde{B}}} D_n$. 
Since the real part of a polynomial tends to infinity approximately on the half of the rays, 
there exists an angle $\Gamma$ of positive size such that 
two of the expressions $|R_{\alpha_j}(z)|e^{\rea \tilde Q_j(z)}$ have the same asymptotics 
inside the angle, say, $|R_{\alpha_1}(z)|e^{\rea \tilde Q_1(z)} \asymp 
|R_{\alpha_2}(z)|e^{\rea \tilde Q_2(z)}$ as $|z|\to \infty$, $z\in \Gamma$. This is obviously 
impossible if the leading coefficients of $\tilde Q_1$ and $\tilde Q_2$ are different. This contradiction 
completes the proof of the proposition.
\end{proof}


\subsection{End of the proof of Theorem \ref{ordering}} 
The rest of the proof is similar to the proof of \cite[Theorem 1.8]{abb}.
Recall that $A=A_1A_2=\tilde{A_1}\tilde{A_2}$. 
Since $A_1$ and $\tilde A_1$ belong to $\he$, we have 
\begin{equation}
\label{bro1}
\sum_{t_n\in T} \frac{|A_1(t_n)|^2}{\mu_n |A'(t_n)|^2} = 
\sum_{t_n \in \mathcal{Z}_{A_2}}  \frac{1}{\mu_n |A_2'(t_n)|^2} <\infty,
\end{equation}
and, analogously,
\begin{equation}
\label{bro2}
\sum_{t_n \in \mathcal{Z}_{\tilde A_2}}  \frac{1}{\mu_n |\tilde A_2'(t_n)|^2} <\infty,
\end{equation}

Assume that (i) in Proposition \ref{ordlemma} holds. Dividing if necessary
$B$ by a polynomial we may assume that
$|B'(t_{n_l})|\leq|\tilde{B}(t_{n_l})|$. Hence, we may construct a
lacunary canonical product $U_1$ such that
$\mathcal{Z}_{U_1}\subset\mathcal{Z}_B$ and
$$|B'(t_n)|\leq|\tilde{B}(t_n)|,\qquad t_n\in\mathcal{Z}_{U_1}.$$
Let $U_2$ be another lacunary product with zeros in $\mathbb{C} 
\setminus \cup_{t_n\in T} D(t_n, C|t_n|^{-N})$ such that
\begin{equation}
|U_2(t_n)|=o(|U_1(t_n)|),\qquad n\rightarrow\infty,\qquad t_n\in
T\setminus\mathcal{Z}_{U_1}, 
\label{Ueq}
\end{equation}
\begin{equation}
\label{Ueq2}
|U_2(t_n)|=o(|U_1'(t_n)|),\qquad t_n\in T. 
\end{equation}
This may be achieved if we choose zeros of $U_2$ to be much sparser
than the zeros of $U_1$. Let us show that in this case
$$
f:=A_1\cdot\frac{U_2}{U_1}\in\he,
$$
which contradicts the localization. Since $A_1$ is in $\he$, while
$U_1$ and $U_2$ are lacunary products, it is clear that 
conditions (ii) and (iii) hold for $f$. It remains to show that
$$
\sum_{t_n\in T}\frac{|f(t_n)|^2}{|A'(t_n)|^2\mu_n}<\infty.
$$
Since $f$ vanishes on $\mathcal{Z}_{A_1}\setminus\mathcal{Z}_{U_1}$, 
we need to estimate
the sums over $\mathcal{Z}_{A_2}$ and $\mathcal{Z}_{U_1}$.
By \eqref{Ueq} and \eqref{bro1}, we have
$$
\sum_{t_n\in\mathcal{Z}_{A_2}}\frac{|f(t_n)|^2}{|A'(t_n)|^2\mu_n}
=\sum_{t_n\in\mathcal{Z}_{A_2}}\frac{|U_2(t_n)|^2}{|U_1(t_n)|^2}
\cdot\frac{1}{|A'_2(t_n)|^2\mu_n}<\infty
$$
To estimate the sum over
$\mathcal{Z}_{U_1}$ note first that $BA_0$ divides
$A=\tilde{B}A_0\tilde{A}_2$, whence $B$ divides $\tilde{A_2}$.
Thus $\mathcal{Z}_{U_1}\subset\mathcal{Z}_{\tilde{A_2}}$. Also,
for $t_n\in\mathcal{Z}_{U_1}$,
\begin{equation}
|A'_1(t_n)|=|B'(t_n)|\cdot|A_0(t_n)|\leq|A_0(t_n)|\cdot|\tilde{B}(t_n)|=|\tilde{A}_1(t_n)|.
\label{Ueq3}
\end{equation}
Now by \eqref{Ueq2}, \eqref{Ueq3} and \eqref{bro2} we have
$$
\sum_{t_n\in\mathcal{Z}_{U_1}}\frac{|f(t_n)|^2}{|A'(t_n)|^2\mu_n}=
\sum_{t_n\in\mathcal{Z}_{U_1}}
\frac{|U_2(t_n)|^2}{|U_1'(t_n)|^2}\cdot
\frac{|A'_1(t_n)|^2}{|\tilde{A}_1(t_n)|^2|\tilde{A}'_2(t_n)|^2\mu_n}
\lesssim \sum_{t_n\in\mathcal{Z}_{U_1}}\frac{1}{|\tilde{A}'_2(t_n)|^2\mu_n}<\infty.
$$
Thus, $f\in\he$ and this contradiction completes the proof of Theorem \ref{ordering}. 
\qed
\bigskip


\section{Localization of type $2$}
In this section we prove Theorem \ref{type2}. In what follows we will need 
the following property of functions in the generalized
Hamburger--Krein class: since $1/F$ is a Cauchy transform and $|F'(z_n)|$ 
decays faster than any power, we have
\begin{equation}
\label{slu}
\sum_{n} \frac{z_n^k}{F'(z_n)}=0, \qquad k\in \mathbb{Z}_+.
\end{equation}
Otherwise, by the arguments from Subsection \ref{pd1}, $1/F$ 
decays at most polynomially away from the zeros whence $F$ itself is a polynomial. 
Furthermore, it follows from \eqref{slu} that for any $K, M>0$ 
\begin{equation}
\label{slu1}
|z|^M = o(|F(z)|), \qquad |z|\to\infty, \ \ {\rm dist}\, (z, \{z_n\}) \ge |z|^{-K}. 
\end{equation}
\medskip

\subsection{Proof of sufficiency in Theorem \ref{type2} \label{2suf}} 
Assume that $\he$ satisfies the conditions (i)--(iii).
We will show that in this case $\he$ has 
localization of type $2$. Let $T=T_1\cup T_2$ and let $A=A_1A_2$, where
$A_2$ is the Hamburger--Krein class function from (i).
Let $\mathcal{H}_2$ be the Cauchy--de Branges space constructed from
$T_2$ and $\mu|_{T_2}$, i.e., $\mathcal{H}_2= \mathcal{H}(T_2,A_2, 
\mu|_{T_2})$.

By the hypothesis, the orthogonal complement $\mathcal{L}$ to the polynomials 
in $L^2(T_2,\mu|_{T_2})$ is finite-dimensional. 
If $\{d_n\} \in L^2(T_2,\mu|_{T_2}) \setminus \mathcal{L}$, then there exists
a nonzero moment for the sequence $\{d_n\}$, that is, 
$\sum_{t_n \in T_2} \mu_n d_n t_n^K \ne 0$ for some $K\in\mathbb{N}_0$.
If $f(z) = A_2(z) \sum_{t_n \in T_2} \frac{\mu_n d_n}{z-t_n}$ is the corresponding
function from $\ho_2$, then, for any $M>0$, the function $f$
has a zero in $D(t_n, |t_n|^{-M})$, $t_n \in T_2$,  
when $n$ is sufficiently large (see Subsection \ref{pd1}). 
Thus, for any function in 
$\ho_2$ except some finite-dimensional subspace, its zeros are localized near the whole set $T_2$. 

Now let $\mathcal{G}$ be the subspace of the Cauchy--de Branges space 
$\mathcal{H}_2$ defined by 
$$
\mathcal{G} = \bigg\{A_2\sum_{t_n \in T_2} \frac{\mu_n d_n}{z-t_n}:\, \{d_n\}
\in \mathcal{L} \bigg\}.
$$
This is a finite-dimensional subspace of $\mathcal{H}_2$ 
and it is easy to see that $F\in \mathcal{G}$ if and only if
$F \in \mathcal{H}_2$ and, for any $M>0$,  
$|F(z)/A_2(z)| = o(|z|^{-M})$, as $|z|\to \infty$ outside a set of zero density
(see \eqref{or} and \eqref{LReq}). 
Thus, $\mathcal{G}$ is a finite-dimensional space of entire functions
with the division property and so it consists of the functions 
of the form $SP$ where $S$ is some fixed zero-free function
and $P$ is any polynomial of degree less than some fixed number $L$. 
Note that if $SP\in\mathcal{G}$ 
and so $SP/A_2$ decays faster than any power away from zeros of $A_2$,
then we may conclude that $A_2/S$ also is a function in Hamburger--Krein class. 
Replacing $A_2$ by $A_2/S$ we may 
assume that $\mathcal{G}$ consists of polynomials. 

We conclude that $\mathcal{H}_2$ has the localization
property and for any $F\in\mathcal{H}_2$ we either have
$T_F=\emptyset$ (i.e., $F$ is a polynomial) or $T_F=T_2$.

Let $f\in\mathcal{H}(T,A,\mu)$, $f(z)=A(z) \sum_{t_n\in
T}\frac{c_n\mu^{1/2}_n}{z-t_n}$. Since, by (iii),
$|A_2(t_n)|\mu^{1/2}_n$ tends to zero faster than any power of
$t_n\in T_1$ when $|t_n|\to \infty$, we have
\begin{equation}
A_2(z)\sum_{t_n\in
T_1}\frac{c_n\mu^{1/2}_n}{z-t_n}=\sum_{t_n\in
T_1}\frac{A_2(t_n)c_n\mu^{1/2}_n}{z-t_n}+H(z) \label{Heq}
\end{equation}
for some entire function $H$ (note that the residues on the left
and the right coincide). Let us show using Theorem \ref{inc} that $H\in\mathcal{H}_2$. Indeed, 
the Cauchy transform on the left-hand side of \eqref{Heq}
is bounded on $T_2$ and so 
$$
\sum_{t_n\in T_2}\frac{|H(t_n)|^2}{|A'_2(t_n)|^2\mu_n}<\infty,
$$
Conditions (ii) and (iii) of Theorem \ref{inc} are fulfilled since
$1/A_2$ is a Cauchy transform whence the same is true for $H/A_2$. 

Note also that $F(z):=A_2(z)\sum_{t_n\in
T_2}\frac{c_n\mu^{1/2}_n}{z-t_n}$ is by definition in $\mathcal{H}_2$. Thus,
$$
f = A_1 (g +H+F) 
$$
where $g(z)= \sum_{t_n\in T_1}\frac{c_nA_2(t_n)\mu^{1/2}_n}{z-t_n}$ 
and $H+F\in\mathcal{H}_2$. 

Assume that $H+F\neq0$. Then, either $H+F$ is a polynomial or the
zeros of $H+F$ are localized near $T_2$ up to a finite set. 
In both cases, there exists $K>0$ 
such that the discs $D(t_k,r_k)$, $t_k\in T$, $r_k=|t_k|^{-K}$, 
are pairwise disjoint and, for sufficiently large $k$ we have 
$$
|H(z)+F(z)|>1, \qquad |z-t_k| = r_k.
$$
In the case when the zeros of $H+F$ are localized 
near the whole set $T_2$ up to a finite set, we use Corollary \ref{maxx}
and \eqref{slu1} applied to $A_2$.
Since, $|g(z)|\to 0$ whenever $|z-t_k|=r_k$ and $k\rightarrow\infty$, we
conclude by the Rouch\'e theorem that $A_1(g+H+F)$ has exactly one zero
in each $D(t_k,r_k)$, $t_k\in T_1$, except possibly a finite
number. Also, if $H+F$ is not a polynomial, then $f$ has zeros near the whole set $T_2$
up to a finite subset (again apply the Rouch\'e
theorem to small disks $D(t_k,r_k)$, $t_k\in T_2$,
$r_k=|t_k|^{-K}$, and use the fact that $|H+F|\gtrsim1$,
$|z-t_k|=r_k$).

It remains to consider the case $H+F=0$, i.e., $f=A_1g$. Since  the polynomials
are dense in $L^2(T,\tilde{\mu})$, the space
$\mathcal{H}(T_1,A_1,\tilde{\mu})$ has the strong localization property,
and so $T_f = T_1$ up to a finite set. 
\medskip


\subsection{Proof of necessity in Theorem \ref{type2} \label{2nec}} 
Assume that $\mathcal{H}(T,A, \mu)$ has the
localization property of type $2$. Let
$f$ be a function from $\mathcal{H}(T,A,\mu)$ such that
$\#(T\setminus T_f)=\infty$. Then, by Lemma \ref{TfLemma} there
exists $T_1$ ($T_1=T_f$ up to a finite set) and 
a function $A_1$ with simple zeros in $T_1$ such that $A_1\in\ho_2$. We
now may write $A=A_1A_2$ for some entire $A_2$ with
$\mathcal{Z}_{A_2}=T_2$. 
\medskip
\\
{\bf Proof of (i).}  Since $A_1\in\he$, we have
$$
\frac{1}{A_2(z)} = 
\frac{A_1(z)}{A(z)} = 
\sum_{t_n\in
T}\frac{c_n\mu^{1/2}_n}{z-t_n}
$$
for some $\{c_n\}\in\ell^2$. It is immediate that $c_n= 1/A_2'(t_n)$, 
$t_n\in T_2$, and $c_n=0$ otherwise. Also we have 
$$
\sum_{t_n\in T_2} \frac{1}{|A_2'(t_n)|^2\mu_n} <\infty.
$$
Since localization property implies that $\mu_n$ decay faster than any power, we conclude that
$A_2$ belongs to the generalized Hamburger--Krein class.
\medskip
\\
{\bf Proof of (ii).} Note that by \eqref{slu} (applied to $A_2$) we have
$$
\sum_{t_n\in T_2} \frac{t_n^k}{A_2'(t_n)}=0, \qquad k\in \mathbb{Z}_+,
$$
whence the sequence $\{(\mu_n A_2'(t_n))^{-1}\}_{t_n\in T_2}$ is orthogonal
to all polynomials in $L^2(T_2, \mu|_{T_2})$. 

Now assume that $\{c_n\} \in\ell^2$ 
and $\{c_n\mu_n^{-1/2}\}$ is orthogonal to all polynomials in $L^2(T_2, \mu|_{T_2})$.
Consider the function $f(z) = A_2(z)\sum_{t_n\in T_2}\frac{c_n\mu^{1/2}_n}{z-t_n}$ 
which belongs to $\mathcal{H}_2= \mathcal{H}(T_2,A_2, 
\mu|_{T_2})$. Since $A_1f \in \he$ and $\he$ has localization property of type 2,
the zeros $\{z_n\}$ of $f$ either form a finite set 
or are localized near $T_2$. 
However, in the latter case $f$ satisfies \eqref{yui0},
a contradiction to the fact that, by \eqref{or}, $f/A_2$ decays faster than any power outside
some set of zero area density. 

Thus, any function $f$ constructed above is of the form $PS$
where $P$ is a polynomial 
and $S$ is some zero-free entire function. It is clear from the localization
property that the function $S$ must be the same for all such $f$-s 
(up to multiplication by a constant), and that $A_2 /S$ also is a Hamburger--Krein
class function. Replacing $A_2$ by $A_2 /S$ we may assume that $f$ is a  polynomial. 

To summarize, for any $\{c_n\mu_n^{-1/2}\}$ which 
is orthogonal to all polynomials in $L^2(T_2, \mu|_{T_2})$, the function $f$
is a polynomial. Since $A_1f\in \he$, it remains to show that the degrees of polynomials 
$P$ such that $PA_1\in \he$ are uniformly bounded. Let us show that the
property that $PA_1\in\he$ for any polynomial $P$ contradicts the
localization property of type 2. If $PA_1\in\he$ for any polynomial $P$,
then the function $A_1(z) \sum_{k\ge 0} a_k z^k$ is in $\he$
for any sequence $\{a_k\}$ such that 
$$
\sum_{k\ge 0} |a_k|\cdot \|z^kA_1\|_{\he} <\infty.
$$ 
This contradicts the localization property since for nonzero
$a_k$ decaying sufficiently rapidly the function 
$A_1(z) \sum_{k\ge 0} a_kz_k$ can not have all but finite number of zeros localized near
$T_1$ or near $T$. 
\medskip
\\
{\bf Proof of (iii).}
Assume that (iii) is not
satisfied, that is, the polynomials are not dense in
$\mathcal{H}(T_1,\tilde{\mu})$, and so this space does not have
the strong localization property. Then there exists $G(z)=A_1(z) \sum_{t_n\in
T_1}\frac{c_nA_2(t_n)\mu^{1/2}_n}{z-t_n}\in\mathcal{H}(T_1,A_1,\tilde{\mu})$,
$\{c_n\}_{t_n\in T_1} \in \ell^2$, 
with the property that there exists an infinite sequence of disks
$D(t_{n_j},|t_{n_j}|^{-M})$, $t_{n_j}\in T_1$, such that
$\#D(t_{n_j},|t_{n_j}|^{-M})\cap \mathcal{Z}_G=0$. Now put
$$
H(z)=A_2(z)\sum_{t_n\in T_1}\frac{c_n\mu^{1/2}_n}{z-t_n}-
\sum_{t_n\in T_1}\frac{c_nA_2(t_n)\mu^{1/2}_n}{z-t_n}.
$$
The function $H$ is entire and, as in the proof of sufficiency,
$H\in\mathcal{H}_2$. This means
that $H$ can be written as
$$
H(z)=-A_2(z)\sum_{t_n\in T_2}\frac{c_n\mu^{1/2}_n}{z-t_n}\quad
\text{ for some } \ \{c_n\}_{t_n\in T_2} \in\ell^2.
$$
Now put $f(z)=\sum_{t_n\in T}\frac{c_n\mu^{1/2}_n}{z-t_n}$. Then $f\in \he$ 
and, by the construction, 
$$
f(z)=-A_1(z)H(z)  + A_1(z)\bigg(\sum_{t_n\in T_1}\frac{c_nA_2(t_n)\mu^{1/2}_n}{z-t_n} + H(z)\bigg) = G(z).
$$
However, the zeros of $g$ are not localized near
the whole $T_1$, a contradiction.
\bigskip


\section{Examples of localization of type 2 \label{exx}}

Here we  give a series of examples of spaces $\mathcal{H}(T,A,\mu)$
with localization of type 2. Clearly, the most subtle part is to satisfy condition (ii)
of Theorem \ref{type2}. However, there exists a standard way to avoid 
completeness of polynomials with finite defect. For similar
constructions see \cite{BS1}.

Let $A$ be an entire function with power separated zero set $T = \{t_n\}$
and of the Hamburger--Krein class. Then, in particular, 
for any $K, M>0$, we have $|A(z)|\gtrsim |z|^M$ when $z\notin \cup_n D(t_n, (|t_n|+1)^{-K})$.
Fix some $N\in\mathbb{N}$ such that $\sum_n |t_n|^{-N} <\infty$, and put
\begin{equation}
\label{mea}
\mu_n = |t_n|^{2N}|A'(t_n)|^{-2}.
\end{equation}
Then the polynomials belong to the space $L^2(\mu)$, $\mu=\sum_n \mu_n\delta_{t_n}$,
but are not dense there. Indeed, for any $k\in\mathbb{N}_0$, we have
$$
\frac{z^{k+1}}{A(z)} = \sum_n \frac{t_n^{k+1}}{A'(t_n) (z-t_n)},
$$
whence $\sum_n c_n \mu_n t_n^k = 0$ (take $z=0$). 
Hence, for $c_n = (A'(t_n) \mu_n)^{-1}$ we have 
$\{c_n\}\in L^2(T, \mu)$. It remains to show that the polynomials have finite 
codimension in  $L^2(T, \mu)$. The following proposition shows that this is
often true. 

Recall that for a positive increasing function $w$ on $\mathbb{R}_+$, 
its Legendre transform $w^\#$ is defined as
$w^\#(x)=\sup_{t\in\mathbb{R}_+} \big(xt-w(t)\big)$. 
If, moreover, $w$ is convex, then $(w^\#)^\# = w$.
In what follows we will use the following technical condition on $w$:
\begin{equation}
\label{tech}
w^\#(x+t) - w^\#(x) \lesssim x^{-1} w(x), \qquad x>1,\  0 \le t \le 1.
\end{equation}
Also recall that a positive increasing function $M$ on $\mathbb{R}_+$ 
is said to be a {\it normal weight} if $w(t) = \log M(e^t)$ is a convex function of $t$.

\begin{proposition}
\label{lezh}
Let $A$ be a Hamburger--Krein class function and let $\mu$ be defined by
\eqref{mea}. Assume that there exist
\begin{itemize}
\item
a finite set of rays $L_j = \{e^{i\theta_j}\}$, $j=1, \dots, J$, 
which divide the plane into a union of angles of size less than $\pi/\rho$, where $\rho$
is the order of $A$\textup; 
\item 
a finite set of positive increasing normal weights $M_j$ on $\mathbb{R}_+$ 
such that the Legendre transforms of the functions 
$w_j(t) = \log M_j(e^t)$ satisfy \eqref{tech},
\end{itemize}
such that, for some $K>0$, 
\begin{equation}
\label{l1}
|A(z)| \le (|z|+1)^K M_j(|z|), \qquad z\in L_j,
\end{equation}
and 
\begin{equation}
\label{l2}
|A'(t_n)| \gtrsim |t_n|^{-K} \max_j M_j(|t_n|), \qquad t_n\in T.
\end{equation}
Then the polynomials have finite \textup(and nonzero\textup) codimension in $L^2(T, \mu)$.
\end{proposition}

\begin{example} {\rm The following functions $A$ 
satisfy the conditions of Proposition \ref{lezh} (if not specified, $A$ is assumed to be 
a zero genus canonical product with zero set $T$):
\begin{itemize}
\item
$t_n = 2^n$, $n\in\mathbb{N}$;
\item
$t_n = n^\alpha$, $n\in\mathbb{N}$, $\alpha>2$; 
\item
$t_n = |n|^\alpha\, {\rm sign}\, n$, $n\in\mathbb{Z}$, $\alpha>1$;
\item 
$A(z) = z^{-1} \sin (\pi z) \sin (\pi i z)$, $T = \mathbb{Z} \cup i\mathbb{Z}$;
\item 
$A(z) = \sigma(z)$, the Weierstrass $\sigma$-function, $T = \mathbb{Z} + i \mathbb{Z}$. 
\end{itemize}  
In all above examples except the first one, one should take 
$w_j(t) = e^{\beta t}$ for some $\beta>0$, whence
$w^\#(x) = \frac{x}{\beta}\big(\log\frac{x}{\beta} - 1\big)$. In the first example 
$w$ and $w^\#$ are quadratic functions. Condition \eqref{tech} is satisfied 
in all these cases. }
\end{example}

\begin{proof}[Proof of Proposition \ref{lezh}]
Assume that $\{c_n\}$ is orthogonal to the polynomials in $L^2(T, \mu)$ and consider the function
$$
f(z) = A(z)\sum_n \frac{c_n\mu_n}{z-t_n}.
$$
We will show that any such function $f$
is a polynomial whose degrees are uniformly bounded above. This will prove the 
proposition.

Note that $f\in \he$, since $\{c_n \mu_n^{1/2} \} \in \ell^2$. It is clear that, for any 
$k\in\mathbb{N}_0$, one has 
$$
z^k f(z) - A(z)\sum_n \frac{c_n\mu_n t_n^k}{z-t_n} \equiv 0.
$$

Since $T$ is power separated, there exists a constant $N_0\in\mathbb{N}$
such that the discs $D_n = D(t_n, |t_n|^{-N_0})$ are pairwise  disjoint
and, from \eqref{l1}, $|A(z)| \le |z|^{-K} M_j(|z|)$ for any $j$
when ${\rm dist}\, (z, L_j) \le 2|z|^{-N_0}$ and $|z|$ is sufficiently large. 

Now fix some ray $L_j = L$ and let $M=M_j$, $w=w_j$. For $z\notin \cup_n D_n$ 
and for any $k\in \mathbb{N}_0$ we have, using \eqref{mea} and \eqref{l2},
$$
\begin{aligned}
|f(z)| &  
\le \frac{|A(z)|}{|z|^k} \sum_n \frac{|t_n|^k |c_n|\mu_n}{|z-t_n|} \le
\frac{|A(z)|}{|z|^k} \sum_n \frac{|c_n| \mu_n^{1/2} |t_n|^{k+N} } {|A'(t_n)|\cdot |z-t_n|} \\
& \lesssim
\frac{|A(z)|}{|z|^k} \sum_n \frac{|t_n|^{k+2N+N_0+K}} {|t_n|^N M(|t_n|)} \lesssim
\frac{|A(z)|}{|z|^k} \sup_n \frac{|t_n|^{k+2N+N_0+K}} {M(|t_n|)}. 
\end{aligned}
$$
Put $m = 2N+N_0 +K$. Then 
$$
\log \sup_n \frac{|t_n|^{k+m}} {M(|t_n|)} = \sup_n \Big( (k+m)\log|t_n| - 
w(\log|t_n|) \Big) \le w^{\#}(k+m). 
$$
Since the estimate for $f$ holds for all $k$, we now have for 
$z\notin \cup_n D_n$ and ${\rm dist}\, (z, L_j) \le 2|z|^{-N_0}$,
$$
\log|f(z)| \le \log M(r) + K\log r + \inf_{k\in\mathbb{N}_0} \Big(w^{\#}(k+m) - k\log r \Big) +O(1),
$$
when $r=|z|$ is sufficiently large. 
It follows from \eqref{tech} that $x\log r - w^\#(x) = y\log r - w^\#(y) + O(\log r)$
whenever $|x-y|\le 1$ and $w^\#(x) \le x\log r$.
Then, obviously, 
$$
\begin{aligned}
\inf_{k\in\mathbb{N}_0} \Big(w^{\#}(k+m) - k\log r\Big) & = m\log r -
\sup_{k\in\mathbb{N}_0} \Big((k+m) \log r - w^{\#}(k+m)\Big) 
\\ & \le
-\sup_{x\ge 0} (x\log r - w^{\#}(x))  +O(\log r) = -w(\log r) + O(\log r),
\end{aligned}
$$
where the constants involved in $O(\log r)$ depend only on $m$ and the constants from \eqref{tech}.
We used that $w$ is convex and so $(w^\#)^\# = w$. Since $w(\log r) = M(r)$, we conclude that
$$
|f(z)| \lesssim (|z|+1)^{N_1}, \qquad  
z\notin \cup_n D_n, \ \ {\rm dist}\, (z, L_j) \le 2(|z|+1)^{-N_0},
$$
where $N_1$ admits a uniform bound. 
It follows that $|f(z)| \lesssim (|z|+1)^{N_1}$ on the ray $L_j$ for any $j$. Now the standard 
Phragm\'en--Lindel\"of principle shows that $f$ is a polynomial of degree at most $N_1$.
\end{proof}

\begin{remark}
{\rm Note that using the same argument one can show that 
in conditions of Proposition \ref{lezh} the polynomials are dense in $L^2(T, \mu)$ 
when $\mu_n = |t_n|^{-N}|A'(t_n)|^{-2}$ and $N$ is sufficiently large.}
\end{remark}

\begin{example}
{\rm One can easily give examples of 
localization of type 2 choosing the measure $\mu_1$ on $T_1$ to be sufficiently small. 
The space $\he$ has localization of type 2 in all cases given below
\begin{enumerate}
\item
Let $T_2 = \{ |n|^\alpha\, {\rm sign}\, n \}_{n\in\mathbb{Z}}$, $\alpha>1$, 
let $A_2(z) = \prod_{t_n\in T_2} (1-z/t_n)$, and let
$T_1 = \{ik^{\beta}\}_{k\ge 1}$, $0<\beta<\alpha$. Put
$$
\mu_n = 
\begin{cases}
|t_n|^{N}|A_2'(t_n)|^{-2},\quad & t_n\in T_2, \\
e^{-|t_n|^\gamma}, & t_n\in T_1,
\end{cases}
$$
where $N>0$ and $\gamma>1/2$ for $0<\beta\le 1/2$, while for $\beta>2$
we assume $\gamma >1/\beta$.
Density of polynomials in $L^2(T_1, \tilde \mu_1)$, 
$\tilde \mu_1 = \sum_{t_n\in T_1} |A_2'(t_n)|^2 \mu_n \delta_{t_n}$,
follows from \cite[Appendix 2]{BS1}.

\item
Let $T_2 = \mathbb{Z} \cup i\mathbb{Z}$, $A(z) = z^{-1} \sin (\pi z) \sin (\pi i z)$, 
and let $T_1 = e^{i\pi/4} (\mathbb{Z} \cup i\mathbb{Z}) \setminus\{0\}$.
For $M, N\in\mathbb{N}$, put 
$$
\mu_n = 
\begin{cases}
|t_n|^{M} e^{-2\pi|t_n|},\quad & t_n\in T_2, \\
|t_n|^{-N} e^{-(2\sqrt{2}+2)\pi|t_n|}, & t_n\in T_1.
\end{cases}
$$

\item
Let $T_2 = \mathbb{Z} + i \mathbb{Z}$, $A(z) = \sigma(z)$,
and let $T_1 = \mathbb{Z} + i \mathbb{Z} + 1/2$.
For $N\in\mathbb{N}$ and $\gamma>2$, put 
$$
\mu_n = 
\begin{cases}
|t_n|^{N} |\sigma'(t_n)|^{-2}, \quad & t_n\in T_2, \\
e^{-|t_n|^\gamma}, & t_n\in T_1.
\end{cases}
$$
\end{enumerate}  }
\end{example}


\end{document}